\documentclass{amsart}
\usepackage{upgreek}
\usepackage{amssymb}
\usepackage{latexsym}
\usepackage{amsmath}
\usepackage{enumerate}
\usepackage{amsthm}
\usepackage{wasysym}
\usepackage{stmaryrd}
\usepackage{mathrsfs}
\usepackage{pifont}
\usepackage{tikz}

\newcommand{\qee} {\hspace*{2mm}\hfill \ding{109}}

\renewcommand{\iff}{\leftrightarrow}

\renewcommand{\phi}{\varphi}

\renewcommand{\Theta}{\varTheta}
\renewcommand{\Phi}{\varPhi}
\renewcommand{\Psi}{\varPsi}
\renewcommand{\Xi}{\varXi}
\renewcommand{\Omega}{\varOmega}
\renewcommand{\Gamma}{\varGamma}

\newcommand{\eq}[1]{\mathop{\sf E}_{\sf #1}}
\newcommand{\quot}[1]{{\sf #1}/{\sf E}_{\sf #1}}
\newcommand{\freg}[2]{#1\mathop{\sf F}#2}
\newcommand{\fregl}[3]{#2\mathop{{\sf F}_{#1}}#3}

\newcommand{\modf}[1]{\widetilde{#1}}
\usetikzlibrary {arrows.meta}

\newtheorem{theorem}{Theorem}[section]
\newtheorem{define}[theorem]{Definition}

\newtheorem{exa}[theorem]{Example}
\newenvironment{example}{\begin{exa} \rm}{\qee\end{exa}}
\newtheorem{exerc}[theorem]{Exercise}

\newtheorem{conj}[theorem]{Conjecture}

\newtheorem{ques}[theorem]{Open Question}

\newtheorem{lemma}{Lemma}[section]

\newtheorem{corollary}{Corollary}[section]
\newtheorem{rem}[theorem]{Remark}
\newenvironment{remark}{\begin{rem} \rm}{\qee\end{rem}}

\newcommand{\To}{\Rightarrow}

\newcommand{\verz}[1]{\{ #1 \}}
\newcommand{\bin}{\mathbin{\in}}
 \newcommand{\tupel}[1]{{\langle #1 \rangle}}
 \newcommand{\mc}[1]{\mathcal #1}

\title{When Bi-interpretability implies Synonymy}

\author{Harvey M. Friedman$^{\dag,\ddag}$}
 \address{Department of Mathematics, Mathematics Building,
         231 West 18th Avenue,
         Columbus, OH 43210, USA
}
\email{friedman@math.ohio-state.edu}

\author{Albert Visser$^\ddag$}
 \address{Philosophy, Faculty of Humanities,
                Utrecht University,
               Janskerkhof 13,
               3512 BL~~Utrecht, The Netherlands}
\email{a.visser@uu.nl}

\keywords{Interpretations, Interpretability, Schr\"oder-Bernstein Theorem}

\subjclass[2010]
{03A05, 
03B30, 
03F25
}

\thanks{$^\dag$Harvey M. Friedman,  
Distinguished University Professor of Mathematics, Philosophy, and Computer Science, Emeritus,
Ohio State University, Columbus, Ohio 43235.
This research was partially supported by the John Templeton Foundation grant ID $\# 36297$. 
The opinions expressed here are those of the author and do not necessarily reflect the views of the John Templeton Foundation.}

\thanks{$^\ddag$We thank Tonny Hurkens who simplified our proof of the Schr\"oder-Bernstein Theorem.
We present his version here with his permission.
We thank Allan van Hulst who verified our proof of the Schr\"oder-Bernstein Theorem in Mizar.
We are grateful to Peter Aczel, Ali Enayat and Wilfrid Hodges for stimulating conversations or e-mail correspondence.
We are grateful to Tim Button and to Leszek Ko{\l}odziejczyk who independently spotted a mistake in
an earlier version of this paper.}

\begin{document}

\begin{abstract}
Two salient notions of sameness of theories are {\em synonymy}, aka {\em definitional equivalence}, and
{\em bi-interpretability}. Of these two {\em definitional equivalence} is the strictest notion. In which cases can
we infer synonymy from bi-interpretability? We study this question for the case of sequential theories.
Our result is as follows. Suppose that two sequential theories are bi-interpretable and that the interpretations
involved in the bi-interpretation are one-dimensional and identity preserving. 
Then, the theories are synonymous. 

The crucial ingredient of our proof is a version of the Schr\"oder-Bernstein theorem 
under very weak conditions. We think this last result has some
independent interest.

We provide an example
to show that this result is optimal. There are two finitely axiomatized sequential theories that are bi-interpretable but not synonymous, 
where precisely one of the interpretations involved in the bi-interpretation is not identity preserving.   
\end{abstract}

\maketitle

\section{Introduction}
When are two theories the same? Are there reasonable ways of abstracting away from the precise choice
of the signature? The notions of {\em synonymy} (or: {\em definitional
equivalence})  and {\em bi-interpretability} provide two good answers to these questions.

The notion of {synonymy} was introduced by de Bouv\`ere in 1965. See \cite{bouv:logi65} and \cite{bouv:syno65}. 
It  appears to be the strictest notion of sameness of theories except strict identity of signature and set of theorems.
Two theories $U$ and $V$  are synonymous iff there is a theory $W$ that is both a definitional extension
of $U$ and of $V$ (where we allow the signatures of $U$ and $V$ to be made disjoint). 
Equivalently, $U$ and $V$ are synonymous iff there are interpretations
$K\mathop {:}U\to V$ and $M\mathop{:}V\to U$, such that $V$ proves that the composition $K\circ M$  is
the identity interpretation on $V$ and such that  $U$ proves that the composition $M\circ K$ of is
the identity interpretation on $U$. (Thus, synonymy is isomorphism in an appropriate category ${\sf INT}_0$ of theories 
and interpretations.) There are many familiar examples of synonymy. First, a theory is synonymous with a definitional extension.
For example, Peano Arithmetic with $0$, {\sf S}, $+$ and $\times$ is synonymous with its extension with $<$ and the additional axiom
$x < y \iff \exists z\; x+{\sf S}z =y$. Then, we have examples that are considered by mathematicians to be trivial alternative
formulations like the theory of weak partial order and the theory of strong partial order. Finally, there are more substantial examples.
E.g., Peano Arithmetic, {\sf PA}, is synonymous with an appropriate theory of strings and concatenation.
See \cite{corc:stri74} and, for a slightly weaker classical result, \cite{quin:conc46}.\footnote{In fact, it is simpler to prove that these theories
are bi-interpretable via one-dimensional, identity preserving interpretations. Then, the results of the present paper immediately yield
the synonymy.}

The notion of {\em bi-interpretability}  was introduced by Alhbrandt and Ziegler in 1986. See \cite{ahlb:quas86} and, also,
\cite{hodg:mode93}. 
Two theories $U$ and $V$ are bi-interpretable iff there are interpretations
$K:U\to V$ and $M:V\to U$, such that there is a $V$-definable function $F$, such that $V$ proves that $F$ is an isomorphism between
 $K\circ M$ and
the identity interpretation on $V$ and such that there is a $U$-definable function $G$, such that
 $U$ proves that $G$ is an isomorphism between $M\circ K$ and
the identity interpretation on $U$. (Thus, bi-interpretability is isomorphism in an appropriate category ${\sf INT}_1$ of theories 
and interpretations.) 

In terms of models the notion of bi-interpretability takes the following form.
We note that an interpretation $K:U\to V$ gives us a uniform construction of an internal model $\modf{K}(\mathcal M)$ of $U$ in
a model $\mathcal M$ of $V$. We find that $U$ and $V$ are bi-interpretable iff, there are
interpretations $K:U\to V$ and $M:V\to U$ and formulas $F$ and $G$, such that, for all models
$\mathcal M$ of $V$,  the formula $F$ defines in $\mathcal M$ an isomorphism  between $\mathcal M$ and
$\modf{M}\modf{K}(\mathcal M)$, and, 
for all models
$\mathcal N$ of $U$, the formula $G$ defines in $\mathcal N$ an isomorphism between $\mathcal N$ and
$\modf{K}\modf{M}(\mathcal N)$. 

Bi-interpretability has a lot of good properties. E.g., it preserves  automorphism groups, $\kappa$-categoricity, finite
axiomatizability, etc. Still the stricter notion synonymy preserves more. For example, synonymy preserves the action of the
automorphism group on the domain of the model. Bi-interpretability (without parameters)
does preserve the automorphism group modulo isomorphism but does not necessarily preserve the action on the
domain. See Section~\ref{counter}  for an example illustrating this difference. 
An example of a property of theories that is not preserved under bi-interpretability (not even under definitional
equivalence) is having a computable model. See \cite{pakh:tenn22}.

Surprisingly, it is not easy to provide natural examples of pairs of theories that are bi-interpretable but not synonymous.
For example, Peano Arithmetic {\sf PA} is \emph{prima facie} bi-interpretable with an appropriate 
theory of the hereditarily finite sets. However, on closer
inspection, these theories are also synonymous. See Subsection~\ref{finitesets}.
In Section~\ref{counter}, we give a verified example of two finitely axiomatized sequential theories that
are bi-interpretable but not synonymous.
 
Our interest in this paper is in the relationship between synonymy and bi-in\-ter\-preta\-bi\-lity for a special class of theories,
{\em the sequential theories}. These are theories that have coding of sequences.  Examples of sequential theories are
Buss's theory ${\sf S}^1_2$, Elementary Arithmetic {\sf EA} or {\sf EFA}, $\mathrm{I}\Sigma_1$, {\sf ZF}, {\sf ZFC}. 
We explain in detail what sequential theories are in Section~\ref{sequentiality}. 

We will show that, for {\em identity-preserving} interpretations between sequential theories, synonymy and 
bi-interpretability coincide (Section~\ref{fromto}). In fact the proof works for a
somewhat wider class: {\em the conceptual theories}. This last class was just defined here
to pin down (more) precisely what is needed for the theorem.

 A central ingredient of our proof is the Schr\"oder-Bernstein Theorem that turns out to hold under
 surprisingly weak conditions. We give the proof of the
 Schr\"oder-Bernstein Theorem under these weak conditions in Section~\ref{schrbe}.
 
 \subsection{Historical Remark}
 The first preprint of this paper came out in the Logic Group Preprint Series of Utrecht University in 2014. 
 On August 4, 2023, Albert Visser received an e-mail from Leszek Ko{\l}odziejczyk reporting that he found a counter-example
 to Theorem 5.4 of the preprint and, on September 6, 2023, Visser received an e-mail from Tim Button that he found a mistake
 in the proof of Theorem 5.4 with suggestions on how to fix it. Fortunately, none of the further results of the paper rested
 on the mistaken Theorem, but only on the correct Corollary 5.5. So, we  eliminated the mistaken theorem. 
 Corollary 5.5 of the earlier version is Theorem~\ref{sterkesmurf} of the present version.

\section{Basic Notions}
In this section, we formulate the basic notions employed in the paper.
We keep the definitions here at an informal level. More detailed definitions
are given in Appendix~\ref{definitions}.

\subsection{Theories}

The primary focus in this paper is on one-sorted theories of first order predicate logic
of  relational signature. We take identity to be a logical constant. 
Our official signatures are relational, however, via the term-unwinding algorithm,
we can also accommodate signatures with functions. Many-sorted theories appear
as an auxiliary in the study of one-sorted theories. We will only consider theories with 
a finite number of sorts. 

The results of the paper that are stated for one-sorted theories can be lifted to the many-sorted case in a fairly 
obvious way. We choose to restrict ourselves to the one-sorted case to keep the presentation reasonably light.  

In this paper, no restriction is needed on the complexity of the set of axioms of a theory or on the size of the signature.

\subsection{Interpretations}
We describe the notion of an $m$-dimensional interpretation for a one-sorted language.
An interpretation $K:U\to V$ is given by the theories $U$ and $V$ and a translation $\tau$  from the language
of $U$ to the language of $V$. The translation is given by a domain formula  $\updelta(\vec x\,)$, where $\vec x$ is a sequence
of $m$ variables, and a mapping from
the predicates of $U$ to formulas of $V$, where an $n$-ary predicate
$P$ is mapped to a formula $A(\vec x_0,\ldots,\vec x_{n-1})$, where the $\vec x_j$ are
appropriately chosen pairwise disjoint sequences of $m$ variables. 
We lift the translation to the full language in the obvious way,
making it commute with the propositional connectives and quantifiers, where we relativize the translated quantifiers
to the domain $\updelta$. We demand that $V$ proves all the translations of theorems of $U$.

We can compose interpretations in the obvious way. Note that the composition of an
$n$-dimensional interpretation with an $m$-dimensional interpretation is $m \times n$-dimensional. 

A 1-dimensional interpretation is  {\em  identity preserving} if translates identity to identity.
A 1-dimensional interpretation is {\em unrelativized} if its domain consists of all the objects
of the interpreting theory.  
A 1-dimensional interpretation is  {\em direct} if it is unrelativized
and preserves identity. Note that all these properties are preserved by composition.

Each interpretation $K\mathop{:}U\to V$ gives us an inner model construction that
builds a model $\modf{K}(\mathcal M)$ of $U$ out of a model
$\mathcal M$ of $V$. Note that $\modf{(\cdot)}$ behaves contravariantly.

If we want to use interpretations to analyze sameness of theories,
we will need, as we will see, to be able to say when two interpretations are
`equal'. Strict identity of interpretations is too fine grained. It depends too
much on arbitrary choices like the selection of bound variables.
We specify a first notion of equality between interpretations: two interpretations
are {\em equal}  when the {\em target theory thinks they are}.
Modulo this identification, the operations
identity and composition
give rise to a category ${\sf INT}_0$, where the theories are objects and
the interpretations arrows.\footnote{For many reasons, the choice for
the reverse direction of the arrows would be more natural.
However, our present choice coheres with the extensive tradition 
in degrees of interpretability. So, we opted to adhere to the usual choice.}

Let \textsf{MOD} be the category with as objects classes of models
and as morphisms all functions between these classes.
We define $\textsf{Mod}(U)$ as the class of all models of $U$.
Suppose $K\mathop{:}U\to V$. Then, $\textsf{Mod}(K)$ is the function
from $\textsf{Mod}(V)$ to $\textsf{Mod}(U)$ given by:
$\mathcal M \mapsto \modf{K}({\mathcal M})$.
It is  clear that {\sf Mod} is a {\em contravariant functor} from ${\sf INT}_0$
to \textsf{MOD}.

\subsection{Sameness of Interpretations}

For each sufficiently good notion of sameness of interpretations there is
 an associated category of theories and interpretations: the category of interpretations
 modulo that notion of sameness.
 Any such a category gives us a notion of isomorphism of theories which can function as a notion
 of sameness.
 
We present a {\em basic list} of salient notions of sameness. For all items in the list
it is easily seen that sameness is preserved by composition. Our list does not have any pretence of
being complete. For example, we omitted notions of sameness based on Ehrenfeucht games.

\subsubsection{Equality}

The interpretations $K,K':U\to V$ are equal when
$V$ `thinks' $K$ and $K'$ are identical. 
By the Completeness Theorem, this is equivalent to saying that, for all $V$-models $\mathcal M$,
$\modf{K}({\mathcal M}) = \modf{K'}({\mathcal M})$. This notion gives rise to the category ${\sf INT}_0$.
Isomorphism in ${\sf INT}_0$ is {\em synonymy} or {\em definitional equivalence}.

\subsubsection{i-Isomorphism}

An i-isomorphism between interpretations $K,M:U\to V$ is given by a $V$-formula
$F$. We demand that $V$ verifies that ``$F$ is an isomorphism between $K$ and $M$'', or, equivalently,
that, for each model $\mathcal M$ of $V$, the function $F^{\mathcal M}$ is an isomorphism between
$\modf{K}(\mathcal M)$ and $\modf{M}(\mathcal M)$. 

Two interpretations $K,K':U\to V$, are {\em i-isomorphic} iff
 there is an i-isomorphism between
$K$ and $K'$. Wilfrid Hodges calls this notion: {\em homotopy}.
See \cite{hodg:mode93}, p222.
 
We can also define the notion of being i-isomorphic semantically.  
The interpretations $K,K':U\to V$, are {\em i-isomorphic} iff
 there is  $V$-formula $F$ such that for all $V$-models
$\mathcal M$, the relation $F^{\mathcal M}$ is an isomorphism between
 $\modf{K}({\mathcal M})$ and $\modf{K'}({\mathcal M})$. 
 
 In case the signature of $U$ is finite, being i-isomorphic has a third characterization.
The interpretations  $K,K':U\to V$, are {\em i-isomorphic} iff, for every $V$-model $\mathcal M$, there is an
 $\mathcal M$-definable isomorphism between $\modf{K}({\mathcal M})$ and 
 $\modf{K'}({\mathcal M})$. (See Theorem~\ref{isoloc}.) 

Clearly, if $K,K'$ are equal in the sense of the previous subsection, they will be i-isomorphic.
The notion of i-isomorphism give rise to a category of interpretations modulo i-isomorphism.
We call this category ${\sf INT}_1$. Isomorphism in ${\sf INT}_1$ is {\em bi-interpretability}.

  \subsubsection{Isomorphism}
  
  Our third notion of sameness of the basic list is  that $K$ and $K'$ are the same if,
  for all
  models $\mathcal M$ of $V$, the internal models
  $\modf{K}({\mathcal M})$ and $\modf{K'}({\mathcal M})$ are isomorphic.
  We will  simply say that $K$ and $K'$ are isomorphic.
  Clearly, i-isomorphism implies isomorphism. We call the associated
  category ${\sf INT}_2$. Isomorphism in ${\sf INT}_2$ is {\em iso-congruence}.
     
  \subsubsection{Elementary Equivalence}
  
   The fourth notion is to say that two interpretations $K$ and $K'$ are the same if, for each $\mathcal M$,
   the internal models $\modf{K}({\mathcal M})$ and $\modf{K'}({\mathcal M})$ are elementarily equivalent.
   We will say that $K$ and $K'$ are elementarily equivalent.
   By the Completeness Theorem, this notion can be
   alternatively defined by saying that $K$ is the same as $K'$ iff, 
   for all $U$-sentences $A$, we have $V\vdash A^K\iff A^{K'}$.
   
     We call the associated
  category ${\sf INT}_3$. Isomorphism in ${\sf INT}_3$ is {\em elementary congruence}
  or {\em sentential congruence}.

      \subsubsection{Identity of Source and Target}
    
    Finally, we have the option of abstracting away from the specific identity
    of interpretations completely, simply counting any two interpretations $K,K':U\to V$
    the same.  The associated category is ${\sf INT}_4$. This is simply the
    structure of degrees of one-dimensional interpretability.
    Isomorphism in ${\sf INT}_4$ is {\em mutual interpretability}.
    
  \subsection{The Many-sorted Case}
    
   Interpretability can be extended to interpretability between many-sorted theories. However to do
   that properly, we would need to develop the notion of piecewise interpretation. Since this notion
   is not needed in the present paper, we just describe interpretations of many-sorted theories
   in one-sorted theories. These are precisely what one would expect: the interpretation $K$ does not
   specify just one domain, but, for each sort $\mathfrak a$, a domain $\updelta_{\mathfrak a}$. We allow
   a different dimension for each sort. The translation
   of a quantifier $\forall x^{\mathfrak a}$ is $\forall \vec x\,(\updelta_{\mathfrak a}(\vec x\,) \to {\ldots})$. We translate
   a predicate $P$ of type ${\mathfrak a}_0,\ldots,{\mathfrak a}_{n-1}$ to a 
   formula $A(\vec x_0,\ldots \vec x_{n-1})$, where the target theory verifies, for $i<n$, the formula
   $A(\vec x_0,\ldots \vec x_{n-1}) \to \updelta_{{\mathfrak a}_i}(\vec x_i)$.\footnote{Note that 
   the sequence $\vec x_i$ has as length the dimension associated to the sort $\mathfrak a_i$.}

   We will consider theories with a designated sort $\mathfrak o$ of objects. An interpretation of
   such a theory into a one-sorted theory is \emph{$\mathfrak o$-direct} iff it is one-dimensional for sort
   $\mathfrak o$, and has $\updelta_{\mathfrak o}(x) := (x=x)$ and translates identity on $\mathfrak o$
   to identity simpliciter. In other words, the interpretation is direct when we restrict our attention to the
   single sort $\mathfrak o$.
   
   \subsection{Parameters}
   
   We can extend our notion of interpretation to {\em interpretation with parameters} as follows.
   Say our interpretation is $K:U\to V$.
   In the target theory, we have a parameter domain $\alpha(\vec z\,)$, which is $V$-provably non-empty.
   The definition of interpretation remains the same but for the fact that the parameters $\vec z$.
   Our condition for $K$ to be an interpretation becomes: 
  \[U\vdash A \;\; \To \;\;  V\vdash \forall \vec z\;(\alpha(\vec z\,) \to A^{K,\vec z}\,).\]
  
  \noindent We note that an interpretation $K:U\to V$ with parameters provides a parametrized
{\em set}  of inner models of $U$ inside a model of $V$. 
    
    \section{Sequentiality and Conceptuality}\label{sequentiality}

We are interested in theories {\em with coding}. There are several `degrees' of coding, like pairing, sequences,
etcetera. We want a notion that allows us to build arbitrary sequences of all objects of our domain. The relevant
notion is {\em sequentiality}. We also define a wider notion  {\em conceptuality}. This last notion is proof-generated: it gives
us the most natural class of theories for which our proof works. All sequential theories are conceptual, but not vice versa.

We have a simple and elegant definition of sequentiality. 
A theory $U$ is {\em sequential} iff it directly interprets {\em adjunctive set theory}
{\sf AS}. Here {\sf AS} is the following theory in the language with only one binary relation
symbol.
\begin{enumerate}[{\sf AS}1.]
\item
$\vdash \exists x\,\forall y\;\, y\not\in x$,
\item
$\vdash \forall x,y \,\exists z\, \forall u\,(u\in z \iff (u\in x \vee u=y))$.
\end{enumerate}

\noindent 
So the basic idea is that we can define a predicate $\in^\star$ in $U$ such that
$\in^\star$ satisfies a very weak set theory involving {\em all} the objects of $U$.
Given this weak set theory, we can develop a theory of sequences for all the objects
in $U$, which again gives us partial truth-predicates, etc. In short, the notion of sequentiality
explicates the idea of a {\em theory with coding}.

\begin{remark}
To develop the notion of sequentiality in a proper way for many-sorted theories we would
need the idea of a {\em piecewise} interpretation. We do not develop the idea of
piecewise interpretation here. Fortunately, one can forget the framework and give the
definition in a theory-free way. It looks like this.
Let $U$ be a theory with sorts $\mathcal S$. The theory $U$ is sequential when we can define, for
each $\mathfrak a,\mathfrak b\in \mathcal S$, a
binary predicate $\in^{\mathfrak{ab}}$ of type $\mathfrak{ab}$ such that:
\begin{enumerate}[a.]
\item
$U \vdash \bigvee_{\mathfrak a\in \mathcal S}\exists x^{\mathfrak a}\,
\bigwedge_{\mathfrak b \in \mathcal S}\forall y^{\mathfrak b}\, y^{\mathfrak b}\not\in^{\mathfrak{ba}} x^{\mathfrak a}$,
\item
$U \vdash \bigwedge_{\mathfrak{a,b} \in \mathcal S}\forall x^{\mathfrak a},y^{\mathfrak b} \,
\bigvee_{\mathfrak c \in \mathcal S}\exists z^{\mathfrak c}\,\bigwedge_{\mathfrak d \in \mathcal S} \forall u^{\mathfrak d}\,
(u^{\mathfrak d}\in^{\mathfrak{dc}} z^{\mathfrak c} \iff (u^{\mathfrak d}\in^{\mathfrak {da}}
 x^{\mathfrak a} \vee u^{\mathfrak d}=^{\mathfrak{db}}y^{\mathfrak b}))$.
\end{enumerate}
Here `$=^{\mathfrak{db}}$' is not really in the language if $\mathfrak d \neq \mathfrak b$. In this case we read 
$u^{\mathfrak d}=^{\mathfrak{db}}y^{\mathfrak b}$ simply as $\bot$.

It's a nice exercise to show that e.g. ${\sf ACA}_0$
and {\sf GB} are sequential. 
 \end{remark}

\noindent
Closely related to {\sf AS} is
 {\em adjunctive class theory} 
 {\sf ac}. We define this theory  as follows.
The theory {\sf ac}  is two-sorted with sorts $\mathfrak o$ (of objects) and $\mathfrak
c$ (of classes). We have identity for every sort and one relation symbol $\in$
between objects and classes, i.e. of type $\mathfrak{oc}$. 
We let $x,y,\ldots$ range over objects and $X,Y,\ldots$ range over classes.
 We have the following axioms
\begin{enumerate}[{\sf ac}1.]
\item
$\vdash \exists X\,\forall x\;\, x \not \in X$,
\item
$\vdash \forall Y,y\,\exists X\;\forall x\;(x\in X \iff (x\in Y \vee x =y))$,
\item
$\vdash X= Y \iff \forall z\;(z\in X \iff z\in Y)$.
\end{enumerate}

\noindent 
Note that extensionality is cheap since we could treat identity on
classes as {\em defined} by the relation of extensional sameness.
The theory {\sf ac} is much weaker than {\sf AS}, since it admits finite
models.
The following theorem is easy to see.

\begin{theorem}
A theory $U$ is sequential iff there is an $\mathfrak o$-direct interpretation of
{\sf ac} in $U$ that is one-dimensional in the interpretation of classes.
\end{theorem}

\noindent
A theory $U$ is {\em conceptual} iff there is an $\mathfrak o$-direct interpretation of {\sf ac} in $U$.
We note that there are conceptual theories that are not sequential. For example, sequential theories always have
infinite domain, but there are conceptual theories with finite models. 
Note also that {\sf AS} is sequential, but  {\sf ac} is {\em not}  conceptual (not even in the appropriate
many-sorted formulation).

For more information on sequentiality and conceptuality, see Appendix~\ref{seqplus}.

\section{The Schr\"oder-Bernstein Theorem}\label{schrbe}

We start with a brief story of the genesis of our version of the Schr\"oder-Bernstein theorem. The first step was taken by
Harvey Friedman who saw that sequential theories should satisfy a
version of the Schr\"oder-Bernstein Theorem\footnote{In this paper we use the version of  the Schr\"oder-Bernstein Theorem
that is formulated in terms of injections. We note that the theorem really should be called: the
Dedekind-Cantor-Schr\"oder-Bernstein Theorem.}, which would lead to the desired result on the coincidence of synonymy and bi-interpretability.\footnote{Harvey
Friedman announced the result in an e-mail of January 1, 2009 starting with the words: \emph{I am guessing that you will not believe in this.}}
Albert Visser
subsequently wrote down a proof, discovering that one needs even less than sequentiality: the thing to use is
adjunctive class theory. Allan van Hulst verified Visser's version of the proof in 
Mizar as part of his master's project under
Freek Wiedijk in 2009.
After hearing a presentation by Allan van Hulst, Tonny Hurkens found a
simplification of the proof. Hurkens proof is shorter and conceptually simpler.
In our presentation here we include Hurkens' simplification. We thank Tonny for his gracious
permission to do so.

We work in the theory {\sf SB} which is  {\sf ac} extended with
two unary predicates on objects: {\sf A} and  {\sf B} and four binary predicates
on objects: $\eq{A}$, $\eq{B}$, {\sf F}, {\sf G}, plus axioms expressing 
that $\eq{A}$ is an equivalence relation on {\sf A}, 
$\eq{B}$ is an equivalence relation on {\sf B}, {\sf F} is an injection from
$\quot{A}$ to $\quot{B}$, {\sf G} is an injection from
$\quot{B}$ to $\quot{A}$. 
We construct a formula {\sf H} that {\sf SB}-provably defines a
bijection between $\quot{A}$ and $\quot{B}$.

We will employ the usual notations like: $\emptyset$, $\verz{x_0,\ldots,x_{n-1}}$,
$\subseteq$.

Our definition of what it means that {\sf F} is a function includes: if
$x\eq{A}x'{\sf F}y' \eq{B} y$, then $x{\sf F}y$. Similarly for {\sf G}.
We will treat {\sf A} as a virtual class and write $x\in{\sf A}$, etc.
\begin{itemize}
\item
A pair of classes $(X,Y)$ is {\em  downwards closed}  if, 
\begin{enumerate}[1.]
\item
$X\subseteq {\sf A}$ and $Y \subseteq {\sf B}$,
\item
 if  $v\mathop{\sf G}u $ and $u\in X$, then
there is a  $v'\in Y$ such that $v'\mathop{\sf G} u $, 
\item
if  $u\mathop{\sf F}v$ and $v\in Y$, then
there is a $u' \in X$ such that $u' \mathop{\sf F} v$.
\end{enumerate}
\item 
We say that $(X,Y)$ is an {\em $x$-switch} if 
\begin{enumerate}[i.]
 \item
   $(X, Y)$ is downwards closed,
\item  $x$ is a member of $X$,
\item
 each member of $X$ is in the range of {\sf G}.
 \end{enumerate}
\item
$x\mathop{\sf H}y$ iff (there is no $x$-switch
and $x\mathop{\sf F}y$) or (there is an $x$-switch and $y\mathop{\sf G}x$).
\end{itemize}

\begin{lemma}[{\sf SB}]
{\sf H} is a function from $\quot{A}$ to $\quot{B}$. 
\end{lemma}

\begin{proof}
We prove that {\sf H} preserves equivalences.
Suppose $x\eq{A} x'$, $y\eq{B} y'$ and $x\mathop{\sf H}y$. 
Suppose there is an $x$-switch $(X,Y)$. It is easy to see that
$(X\cup\verz{x'},Y)$ is an $x'$-switch. It is now immediate that
$x'\mathop{\sf H}y'$. Similarly, if we are given an $x'$-switch, 
we may conclude that there
is an $x$-switch. The remaining case where there is neither
an $x$-switch nor an $x'$-switch is again immediate.

We prove that {\sf H} is functional.
Suppose $x\mathop{\sf H}y$ and $x\mathop{\sf H}y'$.
If there is an $x$-switch, we have $y\mathop{\sf G}x$ and $y'\mathop{\sf G}x$.
So we are done by the injectivity of {\sf G}.
If there is no $x$-switch, we have $x\mathop{\sf F}y$ and $x\mathop{\sf F}y'$.
So we are done by the functionality of {\sf F}.

We prove that {\sf H} is total on {\sf A}.
 If there is no $x$-switch, we are done.
If there is an $x$-switch then $x$ is in the range of {\sf G}, and we are again done.
\end{proof}

\begin{lemma}[{\sf SB}]
 {\sf H} is injective from $\quot{A}$ to $\quot{B}$.
 \end{lemma}
 
 \begin{proof}
Suppose $x\mathop{\sf H}y$ and $x'\mathop{\sf H}y$. If in both cases the same clauses
in the definition of {\sf H} are active, we are easily done.

Suppose there is no $x$-switch and there is an $x'$-switch, say $(X',Y')$.
By the definition of {\sf H}, we have $x\mathop{\sf F}y\mathop{\sf G}x'$. 
It follows that  $x\eq{A} x''$, for some $x''$ in $X'$.
Hence, $(X'\cup\verz{x},Y')$ is an $x$-switch. A contradiction.
\end{proof}

\begin{lemma}[{\sf SB}]
{\sf H} is surjective.
\end{lemma}

\begin{proof} 
Consider any $y\in {\sf B}$. 
First, suppose  $y$ is not in the range of {\sf F}. Let $y\mathop{\sf G}x$.
Then, $(\verz{x},\verz{y})$  is an $x$-switch,
and we have $x\mathop{\sf H}y$. 

Next, suppose $x\mathop{\sf F}y$ and there is no $x$-switch.
In this case $x\mathop{\sf H}y$.

Finally, suppose $x\mathop{\sf F}y$ and there is an $x$-switch $(X,Y)$. Let $y\mathop{\sf G}x_1$.
Then, we have $(X\cup\verz{x_1}, Y\cup\verz{y})$ is an $x_1$-switch. Hence,
$x_1\mathop{\sf H}y$. 

Thus, in all cases $y$ is in the image of {\sf H}. 
\end{proof}

\noindent 
We have proved the following theorem.

\begin{theorem}
In {\sf SB} we can construct a function {\sf H} that is a bijection between
${\sf A}/{\sf E}_{\sf A}$ and ${\sf B}/{\sf E}_{\sf B}$.
\end{theorem}

\begin{example}
Consider a model of {\sf SB}. We write $A$ for the interpretation of {\sf A}, etc. We note that the
function $H$ constructed by the proof of the theorem depends on our choice of classes. Suppose, e.g.,
that our objects are the integers, $A$ is the set of even integers, $B$ is the set of odd integers,
our equivalence relations are identity on the given virtual class,
 $F$ is the successor function domain-restricted to $A$, and $G$ is the successor function domain-restricted to
$B$. In the case that our classes are all possible classes of numbers, the pair of the class of all even numbers  and all odd
numbers is an $a$-switch, for each even $a$. So $H = G^{-1}$, i.e. the predecessor function domain-restricted to $A$.
In the case that our classes are the finite classes, there is no $a$-switch for any even $a$. So, $H=F$. 
\end{example}

\noindent For us the following obvious corollary is relevant.

\begin{corollary}
Let $T$ be a conceptual theory. Suppose we have formulas $Ax$, $By$, $E_A$, $E_B$, $F$, $G$, where
$T$ proves that $E_A$ is an equivalence relation on $\verz{ x \mid Ax}$, that
$E_B$ is an equivalence relation on $\verz{y \mid By}$, that $F$ is an injection from
$\verz{x \mid Ax}/E_A$ to $\verz{y \mid By}/E_B$, and that
$G$ is an injection from $\verz{y \mid By}/E_B$ to $\verz{x \mid Ax}/E_A$.
Then we can find a formula $H$ that $T$-provably defines a bijection between the virtual
classes  $\verz{x \mid A x}/E_A$ and $\verz{ y \mid By}/E_B$.

Our corollary can be rephrased as follows. Suppose that $T$ is conceptual 
and we have one-dimensional interpretations $K,M:{\sf EQ}\to T$, where
{\sf EQ} is the theory of equality. Suppose further that $F:K\to M$ and $G: M \to K$ are injections. Then we can find a bijection
$H:K\to M$, and, thus, $K$ and $M$ are i-isomorphic.

We note that if $T$ is sequential, we can drop the demand that $K$ and $M$ are one-dimensional, since
every interpretation in a sequential theory is i-isomorphic with a one-dimensional interpretation. 
\end{corollary}

\section{From Bi-interpretability to Synonymy}\label{fromto}

In this section we prove our main result. If two theories are bi-interpretable via identity-preserving interpretations, then
they are synonymous.

\begin{theorem}\label{grotesmurf}
Let $U$ and $V$ be any theories and suppose  $K:U\to V$ and $M:V\to U$.
Suppose that, for any model $\mathcal M$ of $V$, we have 
$\widetilde M \widetilde K(\mathcal M)= \mathcal M$
\textup{(}in other words, $K\circ M = {\sf id}_V$ in ${\sf INT}_0$\textup{)}. 
Suppose further that, for any model $\mathcal N$ of $U$, the model
$\widetilde K \widetilde M(\mathcal N)$ is elementarily equivalent to
$\mathcal N$ \textup(in other words, $M\circ K = {\sf id}_U$ in ${\sf INT}_3$\textup). Then $U$ and $V$ are synonymous.

Here is a different formulation: if $K,M$ witness that $V$ is an ${\sf INT}_0$-retract of $U$ and that $U$ is 
an ${\sf INT}_3$-retract of $V$, then $U$ and $V$ are synonymous.
\end{theorem}

\begin{proof}
Consider any model $\mathcal N$ of $U$. 
We have $\widetilde K \widetilde M\widetilde K\widetilde M(\mathcal N) = 
\widetilde K \widetilde M(\mathcal N)$. Let $\mathcal P :=\widetilde K \widetilde M(\mathcal N)$.
So,  $\widetilde K \widetilde M(\mathcal P) = \mathcal P$.
We note that the identity of $\widetilde K \widetilde M(\mathcal P)$ and $\mathcal P$
is witnessed by such statements as $\forall x\,\updelta_{M\circ K}(x)$ and $\forall \vec x\,(P_{M\circ K}\vec x \iff P\vec x\,)$.
Since $\mathcal P$ is 
elementarily equivalent to $\mathcal N$, we have $\widetilde K \widetilde M(\mathcal N) = \mathcal N$.
So $M\circ K = {\sf id}_U$ in ${\sf INT}_0$.
\end{proof}

\noindent
There is, of course, also a model-free proof of the result. 

\begin{theorem}\label{brilsmurf}
 Suppose that
$K: U \to V$ and $M:V\to U$ witness that $V$ is an ${\sf INT}_1$-retract of $U$ and that $U$ is an
${\sf INT}_3$-retract of $V$. Suppose further that 
 $M$ is direct. Then $U$ and $V$ are synonynous.
\end{theorem}

\begin{proof}
Since $M$ is direct, it follows that, in $V$, we have $\updelta_{K\circ M} = \updelta_K$.
We replace $K$ by a definably isomorphic direct interpretation
$K'$. Suppose $F$ is the promised isomorphism between $K\circ M$ and ${\sf id}_V$.
We take, for $P$ of arity $n$, in the signature of $U$: 
\begin{itemize}
\item 
 $P_{K'}(v_0,\ldots,v_{n-1}) :\iff  
  \exists \vec u_0\bin\updelta_K,\ldots, \exists \vec u_{n-1}\bin\updelta_K \\
   \hspace*{6cm}
 (\bigwedge_{i<n} \vec u_i \mathop{F}  v_i \wedge P_K(\vec u_0,\ldots, \vec u_{n-1}))$. 
 \end{itemize}
Clearly, we have an isomorphism $F':K\to K'$, based on the same underlying formula as $F$.
Hence $K',M$  witness that $U$ is an ${\sf INT}_3$-retract of $V$.

We note that $K'\circ M$ is direct. Suppose $R$ is an $m$-ary predicate of $V$. The theory $T$ proves: 
\begin{eqnarray*}
R_{K'\circ M} (x_0,\ldots, x_{m-1}) & \iff  & (R_M(x_0,\ldots ,x_{m-1}))^{K'} \\
& \iff & \exists \vec y_0\bin \updelta_K,\ldots,\exists \vec y_{m-1}\bin \updelta_K \\
&&
(\bigwedge_{j<m}\vec y_j \mathop{F} x_j \wedge (R_M(\vec y_0,\ldots,\vec y_{m-1}))^K) \\
& \iff & \exists \vec y_0\bin \updelta_K,\ldots,\exists \vec y_{m-1}\bin \updelta_K \\
&&  (\bigwedge_{j<m}\vec y_j \mathop F x_j  \wedge R_{K\circ M}(\vec y_0,\ldots,\vec y_{m-1})) \\
& \iff & R(x_0,\ldots,x_{m-1})
\end{eqnarray*}  
Thus, we find: $K'\circ M = {\sf id}_V$ in ${\sf INT}_0$, in other words, $V$ is an
${\sf INT}_0$-retract of $U$.
We apply Theorem~\ref{grotesmurf} to $K',M$ to obtain the desired result that $U$ and $V$
are synonymous.
\end{proof}

  %
  %
  %

\begin{figure}
\begin{tikzpicture}
\draw[thick](-1,0) circle (2.5);
\draw[thick, fill = lightgray](-1,0) circle (1.5);
\node at (-1,0.3) {$K \models U$};
\node at (-1,-0.3) {$K\circ M \models V$};
\node at (-1,-2) {${\sf id} \models V$};
\draw[thick,<->] (0.1,0) -- (1.2,0);
\node at (0.75,0.4) {$F$};
\end{tikzpicture}
\caption{Illustration of the Proof of Theorem~\ref{brilsmurf}}
\end{figure}
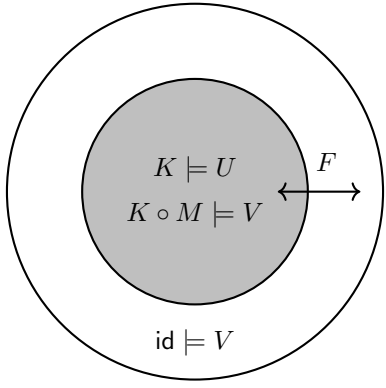

\noindent We are mainly interested in the following corollary.

\begin{corollary}\label{lachendesmurf}
Suppose $U$ and $V$ are bi-interpretable and one of the witnessing
interpretations is direct. Then $U$ and $V$ are synonynous.
\end{corollary}

\noindent We now prove our main theorem.

\begin{theorem}\label{sterkesmurf}
Suppose  $V$ is conceptual and that
$K:U \to V$ and $M:V\to U$ form a bi-interpretation of $U$ and $V$. Let 
$K$ and $M$  both be  identity-preserving.
Then, $U$ and $V$ are synonymous.
\end{theorem}

\begin{proof}
We note that, in $V$, $\updelta_K$ is a (virtual) subclass of the full domain. Hence, we have a definable injection
from  $\updelta_K$ to the full domain.

Again $\updelta_{K\circ M}= \updelta^K_M \cap \updelta_K$
is a (virtual) subclass of $\updelta_K$. Moreover, we have a definable bijection $F$  between  the full domain and $\updelta_{K\circ M}$. 
Hence, we have a definable injection from the full domain into
$\updelta_K$.

We apply the Schr\"oder-Bernstein Theorem to the full domain and $\updelta_K$
providing us with a bijection $G$ between the full domain and $\updelta_K$. 
We define a new interpretation $K':U\to V$, by setting:
\begin{itemize}
\item
 $P_{K'}(v_0,\ldots,v_{n-1}) :\iff \exists w_0\bin \updelta_K,\ldots, \exists w_{n-1}\bin \updelta_K\, \\
 \hspace*{5cm} (\bigwedge_{i<n} v_i 
 \mathop G w_i \wedge P_K(w_0,\ldots,w_{n-1}))$.
 \end{itemize}
Clearly, $K'$ is direct and isomorphic to $K$, so $K'$ and $M$ form a bi-interpretation of $U$ and $V$.
By Corollary~\ref{lachendesmurf}, we may  conclude
that $U$ and $V$ are synonymous.
 \end{proof}
 
  %
    %
  %

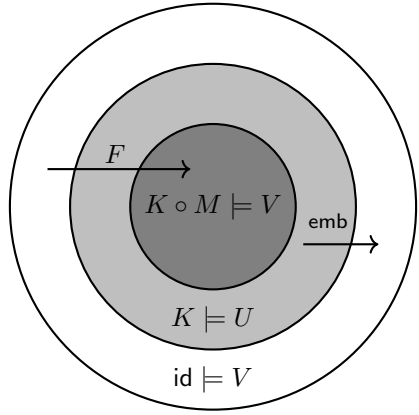
\begin{figure}
\begin{tikzpicture}
\draw[thick](-1,0) circle (2.7);
\draw[thick, fill = lightgray](-1,0) circle (1.9);
\draw[thick, fill = gray](-1,0) circle (1.1);
\node at (-1,0) {$K\circ M \models V$};
\node at (-1,-1.5) {$K\models U$};
\node at (-1,-2.3) {${\sf id}\models V$};
\draw[thick,->] (-3.2,0.5) -- (-1.3,0.5);
\draw[thick,->] (0.2,-0.5) -- (1.2,-0.5);
\node at (-2.3,0.7) {$F$};
\node at (0.53,-0.18) {\footnotesize \sf emb};
\end{tikzpicture}
\caption{Illustration of the Proof of Theorem~\ref{sterkesmurf}}
\end{figure}

\noindent We note that in the circumstances of Theorem~\ref{sterkesmurf}, it follows that $U$ is also conceptual.
 
 
  In Section~\ref{counter}, we provide an example to illustrate
that one cannot drop  the demand of identity preservation for any of the
two interpretations. We provide two finitely axiomatised sequential theories
that are bi-interpretable, but not synonymous. One of the two witnesses of the bi-interpretation
is identity preserving. 

\section{Applications}\label{applications}

In this section we provide a number of applications of Theorem~\ref{sterkesmurf}.

\subsection{Natural Numbers and Integers}

The theory ${\sf PA}^-$ is the theory of the non-negative part of a discretely ordered commutative ring,
See Richard Kaye's book \cite[Chapter 2]{kaye:mode91} and Emil  Je{\v{r}}{\'{a}}bek paper \cite{jera:sequ12}
Let {\sf DOCR} be the theory of a discretely ordered commutative rings. The theories 
${\sf PA}^-$ and {\sf DOCR} are bi-interpretable. The interpretation of  ${\sf PA}^-$ in {\sf DOCR}
is \emph{restriction to the non-negative part}. Kaye uses the well-known pairs construction as interpretation of
{\sf DOCR} in ${\sf PA}^-$. This construction is two-dimensional and employs an equivalence relation.
As shown by Je{\v{r}}{\'{a}}bek in \cite{jera:sequ12}, we have a polynomial pairing function $\tupel{\cdot,\cdot}$
in  ${\sf PA}^-$. We can define our domain as $\tupel{0,x}$, (representing the non-positive numbers)
and $\tupel{x,0}$, for all $x$, representing the non-negative numbers.\footnote{Alternatively, we can use the even and odd numbers
as domain. Note that in ${\sf PA}^-$ these need not be all the numbers. Je{\v{r}}{\'{a}}bek in \cite{jera:sequ12} shows that the odds and the evens
are mutually exclusive in ${\sf PA}^-$.} Since, ${\sf PA}^-$ has the subtraction axiom,
we can define the desired operations as a matter of course. Moreover, the verification that we defined an
 identity preserving  bi-interpretations is immediate. Thus, by Theorem~\ref{sterkesmurf}, we may conclude that  ${\sf PA}^-$ and {\sf DOCR}
 are synonymous.

\subsection{Natural Numbers and Rational Numbers}
 
Julia Robinson, in her seminal paper \cite{robi:defi49},  shows that the natural numbers are definable  in $\mathbb Q$, the field
of the rationals.
This gives us an identity preserving interpretation of $\mathsf{Th}(\mathbb N)$
in $\mathsf{Th}(\mathbb Q)$. Conversely, we can find an identity preserving interpretation
of $\mathsf{Th}(\mathbb Q)$ in $\mathsf{Th}(\mathbb N)$ by using the Cantor pairing and by just considering
 pairs $\tupel{m,n}$ where $m$ and $n$ have no common divisor except 1. Addition and
multiplication are defined in the usual way. We can easily define
internal isomorphisms witnessing that these interpretations form a
identity preserving bi-interpretation. Hence,  $\mathsf{Th}(\mathbb N)$
in $\mathsf{Th}(\mathbb Q)$ are synonymous by Theorem~\ref{sterkesmurf}.

\subsection{Finite Sets and Numbers}\label{finitesets}

We consider the theory ${\sf ZF}^+_{\sf fin}:=(\mathsf{ZF}-\mathsf{INF})+\neg\,\mathsf{INF}+\mathsf{TC}$.
This is {\sf ZF} in the usual formulation minus the axiom of infinity, plus the negation of the axiom
of infinity and the axiom {\sf TC} that tells us that every set has a transitive closure.
Kaye and Wong in their paper  \cite{kaye:inte07} provide a careful verification that
$\mathsf{ZF}^+_{\sf fin}$ and \textsf{PA}
are synonymous. By Theorem~\ref{sterkesmurf}, it is sufficient to show that the
Ackermann interpretation of ${\sf ZF}^+_{\sf fin}$ in {\sf PA} and the von Neumann interpretation of
{\sf PA} in ${\sf ZF}^+_{\sf fin}$ form a bi-interpretation. For further information about the related theory
${\sf ZF}_{\sf fin} = (\mathsf{ZF}-\mathsf{INF})+\neg\,\mathsf{INF}$, see \cite{enay:omeg10}.

\subsection{Sets with or without Urelements}

Benedikt L\"owe shows that a certain version of \textsf{ZF} with a countable set of urelements
is synonymous with \textsf{ZF}. See \cite{loew:sett06}. Again this result is easily obtained
using Theorem~\ref{sterkesmurf}.

\section{Frege meets Cantor: an Example}\label{counter}

In this section, we provide an example of two finitely axiomatized, sequential theories that are bi-interpretable but not
synonymous. One of the two interpretations witnessing bi-interpretability is identity preserving.
The example given is meaningful: it is the comparison of a Frege-style weak set theory and a Cantor-style
weak set theory.\footnote{Since the first preprint of this paper came out in 2014, other separating examples were found.
See, e.g., \cite{enay:cate24}.}

The theory ${\sf ACF}_\flat$ is the one-sorted version of adjunctive class theory with Frege relation.\footnote{It would be
 more natural to give the example for the two-sorted theory {\sf ACF} and to use the notion of piecewise interpretation.
 The definitions of the interpretations and the verification that they form a bi-interpretation would be simpler.
 In fact, we would avoid the use of coding in our definitions.  However, we would need to develop  more of the theory
 of many-sorted interpretations to handle the superior approach smoothly. This is beyond the scope of our present paper.}
Our theory has unary predicates {\sf ob} and {\sf cl},  and
binary  predicates $\in$ and {\sf F}. Here {\sf F} is the Frege relation. We will write $x\mathop{:}{\sf ob}$ for ${\sf ob}(x)$,
$\forall x\mathop{:}{\sf ob} \ldots$ for $\forall x\,({\sf ob}(x) \to \ldots)$, $\exists x\mathop{:}{\sf ob} \ldots$ for $\exists x\,({\sf ob}(x) \wedge \ldots)$.
Similarly for {\sf cl}.
We have the following axioms,
\begin{enumerate}[${\sf ACF}_\flat$1.\/]
\item
$\vdash \forall x\, (x:\mathop{\sf ob} \vee x\mathop{:}{\sf cl})$,
\item
$\vdash \forall x\, \neg\,(x\mathop{:}{\sf ob} \wedge x\mathop{:}{\sf cl})$,
\item
$\vdash \forall x,y\, (x\in y \to (x\mathop{:}{\sf ob} \wedge y\mathop{:}{\sf cl}))$, 
\item
$\vdash \forall x,y \, (\freg{x}{y} \to (x\mathop{:}{\sf ob} \wedge y\mathop{:}{\sf cl}))$, 
\item
$\vdash \exists x\mathop{:}{\sf cl}\; \forall y\mathop{:}{\sf ob}\;\; y \not\in x$,
\item
$\vdash \forall x\mathop{:}{\sf cl}\;\forall y\mathop{:}{\sf ob}  \;\exists z\mathop{:}{\sf cl}
\;\forall w\mathop{:}{\sf ob}\;  (w\in z \iff (w\in x \vee w=y))$.
\item
$\vdash \forall x,y\mathop{:}{\sf cl}\; (\,\forall z\mathop{:}{\sf ob}\; (z\in x \iff z \in y) \to x=y)$.
 \item
 $\vdash\forall x\mathop{:}{\sf ob}\; \exists y\mathop{:}{\sf cl} \; \freg{x}{y}$,
 \item
 $\vdash\forall x\mathop{:}{\sf ob}\;\forall y,y'\mathop{:}{\sf cl}\; ((\freg{x}{y} \wedge \freg{x}{y'} ) \to y=y')$,
 \item
 $\vdash\forall x\mathop{:}{\sf cl}\; \exists y\mathop{:}{\sf ob}\; \freg{y}{x}$.
 \end{enumerate}
 
 \noindent
We provide  1-dimensional interpretations witnessing that {\sf AS} and ${\sf ACF}_\flat$ are bi-interpretable.
Note that by Theorem~\ref{lolsmurf}, it follows that ${\sf ACF}_\flat$ is sequential.

In the context of {\sf AS}, we write:
\begin{itemize}
\item
${\sf pair}(x,y,z) :\iff \exists u,v\,\forall w\, ((w\in u \iff w = x) \wedge\\
\hspace*{2.35cm}  (w\in v \iff (w=x\vee w=y)) \wedge (w\in z \iff (w=u \vee w = v)))$,
\item
${\sf Pair}(x) :\iff  \exists y,z\,{\sf pair}(y,z,x)$,
\item
$\pi_0(z,x) :\iff \exists y\,{\sf pair}(x,y,z)$, $\pi_1(z,y) :\iff \exists x\,{\sf pair}(x,y,z)$,
\item 
${\sf empty}(x) :\iff \forall y\,y\not\in x$, ${\sf inhab}(x) :\iff \neg\,{\sf empty}(x)$.
\item
$x\approx y :\iff \forall z\,(z\in x \iff z\in y)$.
\end{itemize}
We can verify the usual properties of pairing. The $\pi_i$ are functional on {\sf Pair}. 
We will write them using functional notation. We should remember that they are undefined outside
{\sf Pair}.
We first define an interpretation 
$L:{\sf ACF}_\flat\to {\sf AS}$.
\begin{itemize}
\item
$\updelta_L(x) :\iff  {\sf Pair}(x)$,
\item
${\sf ob}_L(x) :\iff {\sf Pair}(x) \wedge {\sf empty}(\pi_0(x))$,
\item
 ${\sf cl}_L(x) :\iff  {\sf Pair}(x) \wedge {\sf inhab}(\pi_0(x))$,
\item 
$x=_L y :\iff  (x,y\mathop{:}{\sf ob}_L  \wedge \pi_1(x)=\pi_1(y)) \vee
(x,y\mathop{:}{\sf cl}_L\wedge  \pi_1(x) \approx \pi_1(y))$.
\item
$x \in_L y :\iff x\mathop{:}{\sf ob}_L \wedge y\mathop{:}{\sf cl}_L \wedge \pi_1(x)\in \pi_1(y)$,
\item
$ \fregl{L}{x}{y} \mathop{:}\iff x:{\sf ob}_L \wedge y\mathop{:}{\sf cl}_L \wedge \pi_1(x) \approx \pi_1(y)$.
\end{itemize} 

\noindent 
It is easy to see that  the specified translation does carry an  interpretation
of ${\sf ACF}_\flat$ in {\sf AS}, as promised.  
Next, we define an interpretation
  $K:{\sf AS}\to {\sf ACF}_\flat$. 
\begin{itemize}
\item
$\updelta_K(x) :\iff x\mathop{:}{\sf ob}$,
\item
$x =_K y \iff x,y\mathop{:}{\sf ob}  \wedge x=y$,
\item 
$x\in_K y \iff  x,y\mathop{:}{\sf ob} \wedge \exists z\mathop{:}{\sf cl}\, (x\in z \wedge \freg{y}{z} ) $.
\end{itemize}

\noindent 
We note that $K$ is identity preserving.
The verification that our interpretations do indeed specify a bi-interpretation is
entirely routine. For completeness' sake, we provide the computations involved in  Appendix~\ref{verification}.

We show that {\sf AS} and ${\sf ACF}_\flat$ are not synonymous ---not even when we allow parameters. 
 We build the
following model $\mathcal M$ of {\sf AS}.
The domain is inductively specified as the smallest set $M$ such that
if $X$ is a finite subset of $M$, then $\tupel{0,X}$ and $\tupel{1,X}$ are in $M$. 
Let $m,n,\ldots$ range over $M$. We define: $m\in^\star n$ iff $n=\tupel{i,X}$ and
$m\in X$. It is easily seen that $\mathcal M$ is indeed a model of {\sf AS}.
Clearly, for any finite subset $X_0$ of $M$ we can find an automorphism $\sigma$
of $\mathcal M$ of order 2 that fixes $X_0$ and fixes only finitely many elements of $M$.

Suppose {\sf AS} and ${\sf ACF}_\flat$ were synonymous. Let $\mathcal N$ be the internal model
of ${\sf ACF}_\flat$ in $\mathcal M$ given by the synonymy. Say the interpretation is $P$, involving a
 finite set of parameters $X_0$. Let $\sigma$ be an automorphism
 of order 2 om $\mc M$ that fixes $X_0$ and that fixes at most finitely many objects.
 Consider the classes $\verz{p,\sigma p}^{\mathcal N}$, where $p$ is
 in ${\sf ob}^{\mathcal N}$. (Note that $\sigma$ must send $\mathcal N$-objects to $\mathcal N$-objects.)  
 Clearly there is an infinity of such classes. By extensionality these
classes are fixed by $\sigma$. This contradicts the fact that $\sigma$ has only finitely many fixed
points.

It is well known that if two models are bi-interpretable (without parameters) then their automorphism groups
are isomorphic. Our example shows that the {\em action} of these automorphism groups on the elements
can be substantially different.


\appendix

\section{Definitions}\label{definitions}

In this appendix, we provide detailed definitions of translations, interpretations and morphisms
between interpretations. 

\subsection{Translations}

Translations are the heart of our interpretations. In fact, they
are often confused with interpretations, but we will not
do that officially. In practice, it is often convenient to conflate
an interpretation and its underlying translation.

A proto-formula is a $\uplambda$-term $\uplambda \vec x. A(\vec x)$,
where the variables of the formula $A$ are among those in $\vec x$. We will think
of proto-formulas modulo $\alpha$-conversion and we will follow the
conventions of the $\uplambda$-calculus for substitution, where we avoid
variable-capture by renaming the bound variables. The arity is a proto-formula
is the length of $\vec x$. 

We define more-dimensional, one-sorted, one-piece
relative translations without parameters. 
Let $\Sigma$ and $\Theta$ be one-sorted signatures.
A  translation  $\tau:\Sigma\to \Theta$ is given
 by a triple $\tupel{m,\updelta,F}$. Here 
 $\updelta$ will be a $m$-ary proto-formula of signature $\Theta$.
 The mapping $F$  associates to each relation symbol $R$ of
$\Sigma$ with arity $n$ an $m\times n$-ary proto-formula of signature
$\Theta$.

We demand that predicate logic proves  $F(R)(\vec x_0,\ldots,\vec x_{n-1}) \to (\updelta (\vec x_0) \wedge \ldots \updelta (\vec x_{n-1}))$. Of 
course, given any candidate proto-formula $F(R)$ not satisfying the restriction, we can obviously modify it
to satisfy the restriction.

We translate $\Sigma$-formulas 
to $\Theta$-formulas  as follows. 
\begin{itemize}
\item 
$(R(x_0,\ldots,x_{n-1}))^\tau:= F(R)(\vec x_0,\ldots , \vec x_{n-1} )$.\\
{\footnotesize We use sloppy notation here. The single variable $x_i$ of the source language
needs to have
no obvious connection with the sequence of variables $\vec x_i$ of the target language 
that represents it.
We need some conventions to properly handle the association $x_i\mapsto \vec x_i$.
We do not treat these details here.
We demand that the $\vec x_i$ are  fully disjoint when the $x_i$ are different.} 
\item 
$(\cdot)^\tau$ commutes with the propositional connectives;
\item
$(\forall x\,A)^\tau := \forall \vec x\,(\updelta(\vec x\,)\to A^\tau)$;
\item 
$(\exists x\,A)^\tau := \exists \vec x\,(\updelta(\vec x\,)\wedge A^\tau)$.
\end{itemize}

\noindent 
Here are some convenient conventions and notations. 
\begin{itemize}
\item
We write $\updelta_\tau$ for `the $\updelta$ of $\tau$' and $F_\tau$ for `the
$F$ of $\tau$'.
\item
We write $R_\tau$ for $F_\tau(R)$.
\item
We write $\vec x\in\updelta_\tau$ for: $\updelta_\tau(\vec x\,)$.
\end{itemize}
There are some natural operations on translations.
The identity translation ${\sf id}:={\sf id}_\Theta$ is one-dimensional and it is defined by:
\begin{itemize}
\item
$\updelta_{\sf id}:= \uplambda x.(x=x)$,
\item
$R_{\sf id}:= \uplambda \vec x. R\vec x$.
\end{itemize}

\noindent
We can compose relative translations as follows.
Suppose $\tau$ is an $m$-dimensional translation from $\Sigma$ to $\Theta$,
and $\nu$\ is a $k$-dimensional translation from $\Theta$ to $\Xi$.
We define:
\begin{itemize}
\item
We suppose that with the variable $x$ we associate under $\tau$  the sequence $x_0,\ldots, x_{m-1}$
and under $\nu$ we send $x_i$ to $\vec x_i$.\\ 
$\updelta_{\tau\nu}(\vec x_0,\ldots, \vec x_{m-1}) := (\updelta_\nu(\vec x_0) \wedge
\ldots \wedge \updelta_\nu(\vec x_{m-1}) \wedge (\updelta_\tau(x))^\nu)$,
\item
Let $R$ be $n$-ary. Suppose that under $\tau$ we associate with $x_i$ the sequence
$x_{i,0},\ldots,x_{i,m-1}$ and that under $\nu$ we associate with $x_{i,j}$ the sequence $\vec x_{i,j}$.
We take:\\
$R_{\tau\nu}(\vec x_{0,0},\ldots \vec x_{n-1,m-1}) = 
\updelta_\nu(\vec  x_{0,0}) \wedge \ldots \wedge \updelta_\nu(\vec x_{n-1,m-1}) \ \wedge (R_\tau(x_0,\ldots,x_{n-1}))^\nu$.
\end{itemize}

\noindent 
A one-dimensional translation $\tau$ {\em preserves identity} if $(x=_\tau y) = (x=y)$.
A one-dimensional translation $\tau$ {\em is unrelativized} if $\updelta_\tau(x) = (x=x)$.  
A one-dimensional translation $\tau$ is  {\em direct} if it is unrelativized
and preserves identity.
Note that all these properties are preserved by composition.

Consider a model $\mathcal M$ with domain $M$ of signature $\Theta$
and $k$-dimensional  translation $\tau:\Sigma \to \Theta$.
Suppose that $N:=\verz{\vec m\bin M^k\mid \mathcal M \models \updelta_\tau \vec m}$.
Then $\tau$ specifies an internal model
$\mathcal N$ of $\mathcal M$ with domain 
$N$ and with $\mathcal N \models R(\vec m_0,\ldots,\vec m_{n-1}) $ iff 
$\mathcal M \models R_\tau(\vec m_0,\ldots,\vec m_{n-1})$.
We will write $\modf{\tau}(\mathcal M)$ 
 for the internal model of $\mathcal M$ given by $\tau$.
We treat the mapping $\tau,\mathcal M \mapsto \modf{\tau}{\mathcal M}$
as a partial function that is defined precisely if $\updelta_\tau^{\mathcal M}$
is non-empty.
Let $\textsf{Mod}$ or $\modf{(\cdot)}$ be the function that maps $\tau$
to $\modf{\tau}$.
We have:
\[ \textsf{Mod}(\tau\circ \rho)(\mathcal M) 
= (\textsf{Mod}(\rho)\circ\textsf{Mod}(\tau))(\mathcal M).\]
So, \textsf{Mod} behaves contravariantly.

\subsection{Relative Interpretations}\label{relint}

A translation $\tau$ supports a {\em relative interpretation} 
 of a theory $U$ in a theory $V$, if, 
for all $U$-sentences $A$, $U\vdash A \To V\vdash A^\tau$.
Note that this automatically takes care of the theory of identity
and assures us that $\updelta_\tau$ is inhabited.
We will write $K=\tupel{U,\tau,V}$ for the interpretation supported by $\tau$.
We write $K:U\to V$ for: $K$ is an interpretation of the form $\tupel{U,\tau,V}$.
If $M$ is an interpretation, $\tau_M$  will be its second component, so
$M=\tupel{U,\tau_M,V}$, for some $U$ and $V$.

\emph{Par abus de langage}, we write `$\updelta_K$' for
$\updelta_{\tau_K}$; we write `$R_K$'  for $R_{\tau_K}$ and we write `$A^K$' for $A^{\tau_K}$, etc.
Here are the definitions of three central operations on interpretations.
\begin{itemize}
\item
Suppose $T$ has signature $\Sigma$. We define:\\
${\sf id}_T:T\to T$ is $\tupel{T,{\sf id}_\Sigma,T}$.
\item
Suppose $K:U\to V$ and $M:V\to W$. We define:\\
$M\circ K:U\to W$ is $\tupel{U,\tau_M\circ\tau_K,W}$.
\end{itemize}
It is easy to see that we indeed correctly defined interpretations
between the theories specified.

\subsection{Equality of Interpretations}\label{equality}

Two interpretations
are {\em equal}  when the {\em target theory thinks they are}.
Specifically, we count two interpretations $K,K':U\to V$ as equal  if they have the same dimension, say $m$, and:
\begin{itemize}
\item
$V\vdash \forall \vec x\;( \updelta_K(\vec x\,)\iff\updelta_{K'}(\vec x\,))$,
\item
$V \vdash \forall \vec x_0,\ldots,\vec x_{n-1}\bin  \updelta_K \; (R_K(\vec x_0,\ldots,\vec x_{n-1})\iff  R_{K'}(\vec x_0,\ldots,\vec x_{n-1}))$.
\end{itemize}
Modulo this identification, the operations
identity and composition
give rise to a category ${\sf INT}_0$, where the theories are objects and
the interpretations arrows.\footnote{For many reasons, the choice for
the reverse direction of the arrows would be more natural.
However, our present choice coheres with the extensive tradition 
in degrees of interpretability. So, we opted to adhere to the present choice.}

Let \textsf{MOD} be the category with as objects classes of models
and as morphisms all functions between these classes.
We define $\textsf{Mod}(U)$ as the class of all models of $U$.
Suppose $K:U\to V$. Then, $\textsf{Mod}(K)$ is the function
from $\textsf{Mod}(V)$ to $\textsf{Mod}(U)$ given by:
$\mathcal M \mapsto \modf{K}({\mathcal M}) := \modf{\tau_K}({\mathcal M})$.
It is  clear that {\sf Mod} is a {\em contravariant functor} from ${\sf INT}_0$
to \textsf{MOD}.

\subsection{Maps between Interpretations}\label{defmorph}

Consider  $K,M\mathop{:}U\to V$. Suppose $K$ is $m$-dimensional and $M$ is $k$-dimensional.
  A $V$-definable, 
$V$-provable morphism from 
$K$ to $M$ is a triple
$\tupel{K,F,M}$, where $F$ is a $m+k$-ary proto-formula.\footnote{Since, in this stage, we are looking
at definitions without parameters we could, perhaps, better speak of \emph{$V$-0-definable}. Parameters may
be added but in the context where we consider theories rather than models some extra details are needed
to make everything work smoothly.}  
We write $\vec x\mathrel{F}\vec y$ for $F(\vec x,\vec y\,)$.
We demand that $F$ has  the following properties.
\begin{itemize}
\item
$V\vdash \vec x\mathrel{F}\vec y \to (\vec x \in \updelta_K \wedge \vec y \in \updelta_M)$.
\item
$V\vdash \vec x=_K \vec u\mathrel{F} \vec v=_M\vec y \to \vec x\mathrel{F}\vec y$.
\item
$V\vdash \forall \vec x \bin  \updelta_K\,\exists \vec y\bin  \updelta_M\; \vec x\mathrel{F} \vec y$.
\item
$V\vdash (\vec x\mathrel{F}\vec y \wedge \vec x\mathrel{F}\vec z\,) \to \vec y=_M \vec z$. 
\item
$V\vdash (\vec x_0F\vec y_0 \wedge \ldots \wedge \vec x_{n-1}F\vec y_{n-1} \wedge 
R_K(\vec x_0,\ldots,\vec x_{n-1})) \to R_M(\vec y_0,\ldots,\vec y_{n-1})$.

\end{itemize}

\noindent We will call the arrows between interpretations: {\em i-maps}
or {\em i-morphisms}.
We write $F\mathop{:}K\to M$ for: $\tupel{K,F,M}$ is a $V$-provable,
$V$-definable morphism from $K$ to $M$. Remember that 
the theories $U$ and $V$  are part of the data for $K$ and $M$. 
We consider $F,G\mathop{:}K\to M$ as {\em equal}  when they are $V$-provably the same. 

An isomorphism of interpretations is easily seen to be a morphism with the 
following extra properties.
\begin{itemize}
\item
$V\vdash \forall \vec y\bin  \updelta_M\,\exists \vec x\bin \updelta_K\; \vec x\mathrel{F}\vec y$,
\item
$V\vdash (\vec x\mathrel{F}\vec y\wedge \vec z\mathrel{F}\vec y\,) \to \vec x=_K\vec z$,
\item
$V\vdash (\vec x_0F\vec y_0 \wedge \ldots \wedge  \vec x_{n-1}F\vec y_{n-1} \wedge 
R_M(\vec y_0,\ldots,\vec y_{n-1})) \to R_K(\vec x_0,\ldots,\vec x_{n-1})$.
\end{itemize}

\noindent 
We call such isomorphisms: i-isomorphisms.
By a simple compactness argument one may prove:

\begin{theorem}\label{isoloc}
Suppose the signature of $U$ is finite.
Consider $K,M:U\to V$. Suppose that, for every model $\mathcal N$ of
$V$, there is an  $\mathcal N$-definable isomorphism
between $\modf{K}({\mathcal N})$ and $\modf{M}({\mathcal N})$.
Then, $K$ and $M$ are i-isomorphic. 
\end{theorem}

\subsection{Adding Parameters}

We can add parameters in the obvious way. An interpretation $K:U\to V$ with parameters
will have a $k$-dimensional parameter domain $\alpha$ (officially a proto-formula), where $V\vdash \exists \vec x\;\alpha\vec x$.
We allow the extra variables $\vec x$ to occur in the translations of the $U$ formulas. We have to take
the appropriate measures to avoid variable-clashes. The condition for $K$ to be an interpretation changes into:
$V \vdash \forall \vec x\;(\alpha\vec x \to A^{K,\vec x})$, where $A$ is a theorem of $U$.

We note that, in the presence of parameters, the function $\widetilde K$ associates a class of models of $U$ to a model of $V$.

Similar adaptations are needed to define i-isomorphisms with parameters.

\section{Background for Sequentiality and Conceptuality}\label{seqplus}

The notion of sequential theory was introduced by Pavel Pudl\'ak in his 
paper \cite{pudl:prime83}. 
Pudl\'ak uses his notion for the study of
the degrees of local multi-dimensional parametric interpretability.
He proves that sequential theories are prime in this degree structure.
In \cite{pudl:cuts85},  sequential theories provide the right level
of generality for theorems about consistency statements.

The notion of sequential theory was independently invented by Friedman
who called it {\em adequate theory}. See Smory\'nski's survey 
\cite{smor:nons85}.
Friedman uses the notion to provide the Friedman characterization of interpretability
among finitely axiomatized sequential theories. (See also \cite{viss:inte90} 
and \cite{viss:insi92}.) Moreover, he shows that ordinary interpretability and
faithful interpretability among finitely axiomatized sequential theories
coincide. (See also \cite{viss:unpr93} and \cite{viss:faith05}.)

The story of the weak set theory {\sf AS} can be traced in the following papers:
\cite{szmi:mutu52},    \cite{coll:inte70}, \cite{pudl:cuts85}, \cite{nels:pred86},  \cite{mont:mini94},
  \cite{myci:latt90} (appendix III),  \cite{viss:pair08},   \cite{viss:card09}.
  The connection between {\sf AS} and sequentiality is made in \cite{pudl:cuts85} and  \cite{myci:latt90}.

For further work concerning sequential theories, see, e.g.,
\cite{haje:meta91}, \cite{viss:unpr93},
 \cite{viss:over98}, \cite{joos:inte00}, \cite{viss:faith05},  \cite{viss:what13}.
 The paper \cite{viss:what13} gives surveys many aspects of sequentiality.
 
 A theorem that is relevant in  this paper is Theorem 10.7 of \cite{viss:cate06}:
 
 \begin{theorem}\label{lolsmurf}
  Sequentiality is preserved to ${\sf INT}_1$-retracts for one-dimensional interpretations. In other words: if $V$ is sequential and
  if $U$ is a one-dimensional retract in ${\sf INT}_1$ of $V$, then $U$ is sequential. 
  \end{theorem}
  
  \begin{proof}
  Suppose $K:U\to V$ and  $M:V\to U$ are one-dimensional and $M\circ K$ is i-isomorphic to ${\sf id}_U$ via $F$.
  Let $\in^\star$ be the $V$-formula witnessing the sequentiality of $V$.
  We define the $U$-formula $\in^\ast$ witnessing the sequentiality of $U$ by:
  $x\in^\ast y$ iff $y$ is in $\updelta_M$ and, for some $z$ with $z\mathop{F}x$, we have
   $(z\in^\star y)^M$. 
  \end{proof}
  
  \noindent This result  holds only for one-dimensional interpretability.
  There are examples of non-sequential theories that are bi-interpretable with a
  sequential theory.
   Since bi-interpretablity is such a good
  notion of sameness of theories, one could argue that the failure of closure of sequential theories under bi-interpretability
  is a defect and that we need a slightly more general notion to fully reflect the intuitions that sequentiality is intended
  to capture. For an elaboration of this point, see \cite{viss:what13}.

We can easily adapt Theorem~\ref{lolsmurf}, to obtain: 
 
  \begin{theorem}
 Conceptuality is preserved to ${\sf INT}_1$-retracts. 
  \end{theorem}

\section{Verification of Bi-interpretability}\label{verification}

We verify that the interpretations $K$ and $L$ of Section~\ref{counter} do indeed
form a bi-interpretation.
We first compute $M:=(L\circ K):{\sf AS} \to {\sf AS}$. 
We find:
\begin{itemize}
\item
$\updelta_M(x) \iff (\updelta_L(x) \wedge (x\in \updelta_K)^L) \iff  (x:{\sf Pair} \wedge {\sf empty}(\pi_0(x)))$,
\item
We have (using the contextual information that $x$ and $y$ are in $\updelta_M$):
\begin{eqnarray*}
x=_M y & \iff &  (x=_Ky)^L \\ & \iff & \pi_1(x) = \pi_1(y)
\end{eqnarray*}
\item
We have (using the contextual information that $x$ and $y$ are in $\updelta_M$):
\begin{eqnarray*}
x \in_M y & \iff &  (\exists z{:}{\sf cl}\, (x\in z \wedge \freg{y}{z} ))^L \\
& \iff &  \exists z{:}{\sf pair}\,({\sf inhab}(\pi_0(z)) \wedge \pi_1(x)\in\pi_1(z) \wedge \pi_1(y) \approx \pi_1(z)) \\
& \iff & \pi_1(x) \in \pi_1(y)
\end{eqnarray*}
\end{itemize}
Clearly $\pi_1$ is the desired isomorphism from $M$ to ${\sf id}_{\sf AS}$.

In the other direction, let $N:= (K \circ L):{\sf ACF}_\flat\to {\sf ACF}_\flat$.
We first note a simple fact about $K$.
Since {\sf F} is functional on classes
 we will use functional notation for it. We  have, for $u,v:{\sf ob}$,
 \begin{eqnarray*}
 u\approx_K v & \iff & \forall w\,(w\in_K u \iff w \in_K v) \\
  & \iff & \forall w\,(w\in {\sf F}(u) \iff w\in{\sf F}(v)) \\
  & \iff & {\sf F}(u)={\sf F}(v).
 \end{eqnarray*}

\noindent We have:

\begin{itemize}
\item
$\updelta_N(x) \iff (x \in \updelta_K(x) \wedge
(\updelta_L(x))^K) \iff (x:{\sf ob} \wedge {\sf Pair}^K(x))$,
\item
 ${\sf ob}_N(x) \iff  ({\sf empty}(\pi_0(x)))^K$,
 \item
 ${\sf cl}_N(x) \iff  ({\sf inhab}(\pi_0(x)))^K$,
\item
$x =_N y \iff  (({\sf ob}_N(x) \wedge {\sf ob}_N(y) \wedge \pi_1^K(x)=\pi^K_1(y)) \vee \\ 
\hspace*{3cm}
({\sf cl}_N(x) \wedge {\sf cl}_N(y) \wedge {\sf F}(\pi_1^K(x))={\sf F}(\pi^K_1(y))))$,
\item
$x\in_N y   \iff   ({\sf ob}_N(x) \wedge {\sf cl}_N(y) \wedge \pi^K_1(x)\in_K\pi^K_1(y))$.  
\item
$\fregl{N}{x}{y} :\iff ({\sf ob}_N(x) \wedge {\sf cl}_N(y) \wedge {\sf F}(\pi^K_1(x))={\sf F}(\pi^K_1(y)))$. 
\end{itemize}
We define $G:N\to {\sf id}_{{\sf ACF}_\flat}$ as follows:
\begin{itemize}
\item
$x\mathrel{G} y :\iff ({\sf ob}_N(x) \wedge {\sf ob}(y) \wedge \pi_1^K(x)=y) \vee ({\sf cl}_N(x) \wedge {\sf cl}(y) \wedge {\sf F}(\pi_1^K(x)) =y)$.
\end{itemize}
We note that $K$ is identity preserving. Thus, e.g., we find that ${\sf pair}^K$ is a true pairing on {\sf ob}.
It is easy to see that, in ${\sf ACF}_\flat$, the virtual classes ${\sf ob}_N$ and ${\sf cl}_N$ form a partition of
$\updelta_N$. It is now trivial to check that $G$ is indeed an isomorphism.

\end{document}